\documentclass[preprint,12pt]{elsarticle}

\usepackage{amssymb}
 \usepackage{amsthm}

\usepackage{amsmath}
\usepackage{graphicx}

\def\part#1{\mathop{\mathrm{Part \left( #1 \right)}}}
\def\partLE#1{\mathop{\mathrm{Part_{\le} \left( #1 \right)}}}

\renewcommand\[{\begin{equation}}
\renewcommand\]{\end{equation}}

\theoremstyle{plain}
\theoremstyle{plain}

  \theoremstyle{definition}
  \newtheorem{defn}{Definition}
 \theoremstyle{definition}
  \newtheorem{example}{Example}
  \theoremstyle{plain}
  \newtheorem{fact}{Fact}
  \theoremstyle{plain}
  \newtheorem{prop}{Proposition}
  \theoremstyle{plain}
  \newtheorem{cor}{Corollary}

\journal{Journal of Statistical Planning and Inference}

\begin{document}

\begin{frontmatter}




\title{A Test Statistic for Weighted Runs}


\author{Frederik Beaujean\corref{cor1}}
\cortext[cor1]{Corresponding author}
\ead{beaujean@mpp.mpg.de}

\author{Allen Caldwell}
\ead{caldwell@mpp.mpg.de}

\address{
Max Planck Institute for Physics
}

\begin{abstract}
A new test statistic based on  runs of weighted deviations
is introduced. Its use for observations sampled from independent normal
distributions is worked out in detail. It supplements the classic
$\chi^{2}$ test which ignores the ordering of observations
and provides additional sensitivity to local deviations
from expectations.  The
exact distribution of the statistic in the non-parametric case is
derived and an algorithm to compute $p$-values is presented. The
computational complexity of the algorithm is derived employing a novel
identity for integer partitions.

\end{abstract}


\end{frontmatter}



\section{Introduction}

In the course of scientific inference, we are faced with one basic
task: comparing observations and model predictions. Based on this
comparison, the hypothesized model may be either accepted or rejected.
In the latter case usually an improved model is sought. The comparison
between observations and the new model is then repeated until a satisfactory
model has been constructed.

In model validation the goal is to provide quantitative test procedures.
The standard approach consists of defining a scalar function of the
data $D$, called \emph{test statistic }$T\left(D\right)$, such that
a large value of $T$ indicates a large deviation of the data from
the expectations under the hypothesized model $\mathcal{H}$. Correspondingly,
small $T$ is seen as good agreement. Let $T_{obs}$ denote the value
of $T$ observed in the actual data set. In order to facilitate the
interpretation of $T$ (how large is too large?), it is useful to
introduce the $p$\emph{-value}. Assuming  $\mathcal{H}$,
the $p$-value is defined as the tail area probability to randomly
sample a value of $T$ larger than or equal to $T_{obs}$: \[
p\equiv P\left(T\ge T_{obs}\left|\mathcal{H}\right.\right).\]
If  $\mathcal{H}$ is correct and all parameters are fixed,
then $p$ is a random variable with uniform distribution on $\left[0,1\right]$.
An incorrect model will typically yield smaller values of $p$. This
is used to guide model selection. For the same data, different models
will give different $p$. Similarly, a different choice of the test
statistic produces a different $p$ for the same model and data. Why
use different statistics? Because one statistic is sensitive to certain,
but not to all properties of the model.

To illustrate this, recall that in the majority of practical applications
the hypothesis $\mathcal{H}$ describing the set of $N$ observations $D=\left\{ X_{i}\right\} $
is constructed with individual observations $X_{i}\in\mathbb{R}$
considered independent. The discrete scalar index $i$ provides an
ordering for the data. It may represent  time, length, energy $\dots$.
For concreteness, let us assume independent, normally distributed
variables $X_{i}\sim\mathcal{N}\left(\mu_{i},\sigma_{i}^{2}\right)$.
We can write the probability density of the data as

\begin{equation}
P\left(D\left|\mathcal{H}\right.\right)=\prod_{i=1}^{N}P\left(X_{i}\left|\mu_{i},\sigma_{i}^{2}\right.\right)\propto\prod_{i=1}^{N}\exp\left(-\frac{\left(X_{i}-\mu_{i}\right)^{2}}{2\sigma_{i}^{2}}\right)=\exp\left(-\frac{\chi_{N}^{2}}{2}\right),\label{eq:Gaussian likelihood}\end{equation}
 where $\chi_{N}^{2}=\sum_{i=1}^{N}\frac{\left(X_{i}-\mu_{i}\right)^{2}}{\sigma_{i}^{2}}$
appears naturally; it is the most widely used test statistic to probe
$\mathcal{H}$; a large $\chi_{N}^{2}$ translates directly into a
small $P\left(D\left|\mathcal{H}\right.\right)$. Note that $\chi_{N}^{2}/N$
is a measure of the average deviation per observation, but it is blind
to the ordering of the data points.

In this paper, we introduce a test statistic sensitive to local deviations
of the data from expectations within an ordered data set. The test
statistic we propose is valid for data which are expected to have
equal probabilities to be below or above expectations. For concreteness,
we consider the $X_{i}$ normally distributed with known mean and
variance, but the formulation is valid for any symmetric distribution.

Statistics involving \emph{runs}; i.e. sequences of observations that
share a common attribute commonly called a \emph{success}, have drawn
a lot of attention. Good reviews are presented in \cite{mood_distribution_1940,balakrishnan_runs_2002,koutras_book_2002,fu_distribution_2003}.
Most of the early work was centered around independent Bernoulli trials;
  cf. \cite{burr_longest_1961,philippou_successes_1986} and \cite{muselli_simple_1996}.
After the introduction of the \emph{Markov chain imbedding} approach
by \cite{fu_distribution_1994}, runs statistics have been considered
also for more complicated models with Markov dependence  \cite{lou_runs_1996,vaggelatou_length_2003,fu_exact_2003,eryilmaz_results_2006}.
For the case of exchangeable binary trials see \cite{eryilmaz_success_2007}. 
Ref. \cite{makri_success_2007} provides a summary of useful formulae and distributions using
a combinatorial approach.

In this paper we call an observation a \emph{success, }S, if the observed
value exceeds the expected value. Similarly an expected value exceeding
the observation is considered a \emph{failure, }F\emph{. }Obviously
the meaning of success and failure may be reversed, and without loss
of generality we may concentrate on the success runs.
 Using the notation
of \cite{fu_distribution_2003} and counting convention of \cite{mood_distribution_1940},
the simplest test statistics based on runs are the \emph{number of
runs of length exactly }$k$, $E_{N,k}$, and the \emph{length of
the longest run}, $L_{N}$. As an example consider the realization
FSSFS; then $E_{5,1}=1$ and $L_{5}=2$. Observe that both $E_{N,k}$
and $L_{N}$ ignore relevant information: a success is a success no
matter how much $X_{i}$ is bigger than its expected value.

The goal of this paper is to enhance the existing procedures based
on $E_{N,k}$ or $L_{N}$ by introducing a new runs statistic $T$,
similar in spirit to $L_{N}$, which includes that extra information.
For simplicity, we construct the statistic only for \emph{success} runs;
the same steps can be taken to define an analogous statistic
 for \emph{failure} runs as well.
$T$ is formally defined in three steps:
\begin{enumerate}
\item Split the data $\left\{ X_{i}\right\} $ into runs. Keep the success runs
and ignore the failure runs. Denote by $A_{j}=\left\{ X_{j_{1}},X_{j_{2}}\dots\right\} $
the set of observations in the $j$-th success run.
\item Associate a weight with each success run. The weight $w\left(A_{j}\right)$
ought to be chosen such that a large weight indicates large discrepancy
between model and observations. A natural choice of the weight function
is a convenient one-to-one function of the probability (density) of
$A_{j}$ such as $w\left(A_{j}\right)=\left[P\left(A_{j}\left|\mathcal{H}\right.\right)\right]^{-1}$
or $w\left(A_{j}\right)=-2 \log \left(P\left(A_{j}\left|\mathcal{H}\right.\right)\right)$.
\item Choose $T$ as the largest weight: \[
T\equiv \max_{j}w\left(A_{j}\right).\]

\end{enumerate}

We proceed as follows. In sec. \ref{sec:Runs-statistic} we first
derive the general expression for $p=P\left(T\ge T_{obs}\left|\mathcal{H}\right.\right)$
given a model with independent observations and equal probability of success
and failure. The formulation is true for arbitrary weights.
Next we give explicit results
in one concrete example of great importance where $X_{i}\sim\mathcal{N}\left(\mu_{i},\sigma_{i}^{2}\right)$
with $\mu_{i},\sigma_{i}^{2}$ known and  $w\left(A_{j}\right)$ chosen as
 the sum of $\chi^{2}$'s of the samples in $A_{j}$.
 For a large number of observations, $N\gtrsim80$, the evaluation
of the exact expressions for $p$ turns out to be highly demanding
both in terms of computer time and memory, as it scales with the number
of integer partitions. Thus we present a Monte Carlo method that works
even for $N\gtrsim1000$ and compare exact and approximate results.
A selection of critical values of $T$ for common confidence levels
is tabulated.
The power of $T$ is studied in sec. \ref{sec:Example}. Compared to
$\chi^2$, tests based on $T$ are superior in detecting departures
from $\mathcal{H}$. This is demonstrated with a specific but
commonly arising example - the presence of an
unexpected localized peak.
As final remarks, we discuss generalizations of $T$
to non-symmetric uncertainties and composite hypotheses (parameters
fit) in sec. \ref{sec:Conclusion}. In the appendix we introduce integer
partitions in more detail and derive the recurrence relation for integer
partitions needed to analyze the computational complexity required
for computing $p$-values for $T$.

\section{Runs statistic\label{sec:Runs-statistic}}
Let us now make the definition of $T$ explicit in the following
example.
The hypothesis $\mathcal{H}$ for the data $\left\{ X_{i}\right\} ,i=1\dots N$
is formulated as:
\begin{enumerate}
\item All observations $\left\{ X_{i}\right\} $ are independent.
\item Each observation is normally distributed, $X_{i}\sim\mathcal{N}\left(\mu_{i},\sigma_{i}^{2}\right)$.
\item Mean $\mu_{i}$ and variance $\sigma_{i}^{2}$ are known.
\end{enumerate}
 We assume that at least one success, $X_{i}>\mu_{i}$ for some $i\in\left\{ 1,2,\dots N\right\} $,
has been observed. The set of observations $D=\left\{ X_{i}\right\} $
is partitioned into subsets containing the success and failure
runs, keeping only the former and ignoring the latter.
 Let $A_{j}$ denote the subset of the observations of the $j^{th}$
success run, $A_{j}=\left\{ X_{j_{1}},X_{j_{2}}\dots\right\} $.
The weight of the $j^{th}$ success run is then taken to be \begin{equation}
w\left(A_{j}\right)\equiv\chi_{run,\, j}^{2}=\sum_{i}\frac{\left(X_{i}-\mu_{i}\right)^{2}}{\sigma_{i}^{2}},\label{eq:chi2 run}\end{equation}
 where the sum over $i$ is understood to cover all $X_{i}\in A_{j}$.
The test statistic  is the largest weight of any success run \[
T\equiv\max_{j}\chi_{run,\, j}^{2}.\]

\noindent
Our goal is to calculate the $p$-value $p\equiv P(T\ge T_{obs}|N)=1-P(T<T_{obs}|N)$.
Due to the symmetry of the normal distribution, for each observation
the chance of success is
\begin{equation}
P\left(X_{i}\textrm{ is a success}\left|\mathcal{H}\right.\right)=
P\left(X_{i}>\mu_{i}\left|\mathcal{H}\right.\right)=\frac{1}{2}.
\label{eq: symm}
\end{equation}

The following analysis up to \eqref{eq:  intermed. result}
is valid for any  $\mathcal{H}$  such that \eqref{eq: symm} holds.
This symmetric Bernoulli property drastically simplifies the calculation.

The \emph{key idea} is that the set of all sequences of successes
and failures in $N$ Bernoulli trials can be decomposed into equivalence
classes, and $P(T<T_{obs}|N)$ can be expressed as an \emph{expectation
value over inequivalent sequences}.

For our purposes a sequence $\xi$ of length $N$ is sufficiently
characterized by the numbers $n_{1,}\dots n_{N}$ denoting the number
of success runs of length one, $n_{1}$, of length two, $n_{2}\dots$
; we write $\mathbf{n}\left(\xi\right)=\left(n_{1},\dots,n_{N}\right)$.
Two sequences $\xi_{1},\xi_{2}$ of length $N$ are declared equivalent,
if they have the same success runs; i.e.
\begin{equation}
\xi_{1}\sim\xi_{2}\Leftrightarrow\mathbf{n}\left(\xi_{1}\right)=
\left(n_{1},\dots,n_{N}\right)=
\mathbf{n}\left(\xi_{2}\right).
\label{def:equival relation}
\end{equation}

 If the last $n_{N-k},\dots,n_{N}$ are zero they may be omitted.
Reflexivity, symmetry and transitivity of $\sim$ follow immediately.
To illustrate definition \eqref{def:equival relation}, consider the
following example.

Let S {[}F{]} denote a success {[}failure{]}, and consider the sequences
$\xi_{1}=\mbox{SSSFFSFS}$ and  $\xi_{2}=\mbox{FSFSSSFS}$. Both
sequences exhibit two success runs of length one, $n_{1}=2,$ and
one success run of length three, $n_{3}=1$.   Hence $\mathbf{n}\left(\xi_{1}\right)=\left(2,0,1\right)=\mathbf{n}\left(\xi_{2}\right)$,
and the sequences are equivalent, $\xi_{1}\sim\xi_{2}$.

In order to find all inequivalent sequences that need to be accounted
for it turns out to be most useful to fix the number of successes,
$r$, and the number of success runs, $M$, with joint density $P(M,r|N)$.
Thus by the law of total probability

\begin{equation}
P(T<T_{obs}|N)=\sum_{r=1}^{N}\sum_{M=1}^{M_{max}}P(T<T_{obs}|M,r,N)\cdot P(M,r|N).\label{eq: sum over r,M}\end{equation}
The maximum number of success runs, $M_{max}$, for fixed $r$ is determined
as follows: there can be no more success runs than successes, so $M\le r$.
On the other hand, the success runs have to be separated by at least one failure,
hence $M\le N-r+1$. For a fixed number of observations, $N$, we have
$M\le\left\lfloor \frac{N+1}{2}\right\rfloor $. It is easily verified
that the latter condition is implied by the first two, and the constraints
are summarized as \[
M_{max}=\min\left(r,N-r+1\right).\]
The joint distribution $P(M,r|N)$  is conveniently expressed as
\[
P(M,r|N)=\frac{1}{2^{N}-1}\cdot R(M,r|N)\]
where $R\left(M,r|N\right)$ denotes the number of (possibly equivalent)
sequences with $M$ success runs and $r$ successes in $N$ Bernoulli
trials. As an example consider $R\left(1,2 | 3\right)=\left|\left\{ \mbox{SSF},\mbox{FSS}\right\} \right|=2$.
In fact $R\left(M,r|N\right)$ can be calculated efficiently by a
recursive algorithm, but it will be seen to cancel out so that we
have no need to compute it.

With $M,\, r,\, N$ fixed, we can decompose $P(T<T_{obs}|M,r,N)$
into the desired average over inequivalent sequences

\begin{equation}
P(T<T_{obs}|M,r,N)=\sum_{\pi}P(T<T_{obs}|\pi)P(\pi|M,r,N).\label{eq: sum over partitions}\end{equation}
The key observation is that the set of inequivalent sequences $\left\{ \pi\right\} \subset\left\{ \xi\right\} $
is in one-to-one correspondence with the set of \emph{integer partitions}
of $r$ into exactly $M$ summands.

Due to their widespread applicability, the integer partitions have
been studied extensively: \cite{andrews_theory_1998} devoted an entire
book to the partitions. For an online overview we refer to \cite{sloane_-line}.
Efficient algorithms to construct all partitions $\left\{ \pi\right\} $
explicitly are well known; e.g. \cite{knuth_art_2005,adnan_distribution_2007}.
These algorithms scale linearly with the number of partitions. We
refer to the appendix for more details on integer partitions; there
we derive the exact number of sequences needed in calculating $P(T<T_{obs}|N)$.
It grows asymptotically as $\mathcal{O}\left(\frac{1}{N}e^{\sqrt{N}}\right)$.

The probability of one such sequence $\pi$, $P(\pi|M,r,N)$ is just
its multiplicity, $W\left(\pi\right)$, divided by the total number
of elements in $\left\{ \xi\right\} $, which is $R\left(M,r|N\right)$.
The multiplicity is found by basic urn model considerations as the
product of the number of ways to shuffle the success runs and the
number of ways to distribute the failures in between and around the
success runs. While the former is just the multinomial coefficient
\[
\binom{M}{n_{1},\dots,n_{N}},
\]
the latter is obtained as a binomial coefficient. Given $M$ success runs
and $N-r$ failures, $M-1$ failures are needed to separate the success runs,
and the remaining $N-r-M+1$ failures can be allocated freely into
the $M+1$ slots surrounding the success runs. Using Eq. 1 from \cite{makri_shortest_2007}
we obtain\begin{flalign*}
W\left(\pi\right) & =\binom{M}{n_{1},\dots,n_{N}}\cdot\binom{N-r+1}{M}=
\frac{(N-r+1)!}{(N-r+1-M)!\cdot\prod_{l}n_{l}!}\\
 & =\frac{\left(N-r+2-M\right)_{M}}{\prod_{l}n_{l}!}\end{flalign*}
with the Pochhammer symbol defined for positive integers $x,\, n$
as \[
\left(x\right)_{n}\equiv\Gamma\left(x+n\right)/\Gamma\left(x\right)=\left(x+n-1\right)!/\left(x-1\right)!\]

Using the independence of the observations, the probability to observe
a value of $T$ smaller than a fixed $T_{obs}$ in an entire sequence
is just the product of probabilities of finding a weight
$w_l<T_{obs}$ in each
individual success run of length $l$, hence we find at once \[
P(T<T_{obs}|\pi)=\prod_{l}\left[P\left(w_l<T_{obs}|l\right)\right]^{n_{l}}.\]
As an example, consider again the sequence $\mbox{SSSFFSFS}$, with
success runs distribution $\mathbf{n}=\left(2,0,1\right)$, then its contribution
reads\[
P(T<T_{obs}|\pi)=P\left(w_l<T_{obs}|l=1\right)^{2}P\left(w_l<T_{obs}|l=3\right).\]
As an intermediate result we note

\begin{flalign}
P\left(T<T_{obs}\left|N\right.\right) & =\sum_{r=1}^{N}\sum_{M=1}^{M_{max}}\sum_{\pi}P(T<T_{obs}|\pi)\cdot P(\pi|M,r,N)\cdot P(M,r|N)\nonumber \\
 & =\sum_{r=1}^{N}\sum_{M=1}^{M_{max}}\sum_{\pi}\prod_{l}\left[P\left(w_l<T_{obs}|l\right)\right]^{n_{l}}\cdot\frac{(N-r+2-M)_{M}}{\left(2^{N}-1\right)\cdot\prod_{l}n_{l}!}\label{eq:  intermed. result}\\
M_{max} & =\min(r,N-r+1)\nonumber \end{flalign}
Eq. \eqref{eq:  intermed. result} is useful for  generalizations where
$P\left(X_{i}\textrm{ is a success}\left|\mathcal{H}\right.\right)=\frac{1}{2}$
but the individual $X_{i}$ are not normally distributed, since at
this point it is still left open which weight $w_l$ to use in order to quantify
the discrepancy between the model prediction and the observed outcome
of \emph{individual} success runs.

Assuming $X_{i}\sim\mathcal{N}\left(\mu_{i},\sigma_{i}^{2}\right)$,
it is most natural to use the $\chi^{2}$ of each run because it corresponds
directly to the probability density of the data. The additional benefit
of this choice is that $P\left(T<T_{obs}|l\right)$ is known exactly,
it is just the cumulative distribution function of the celebrated
$\chi^{2}$- distribution with $l$ degrees of freedom:\begin{flalign}
P\left(T<T_{obs}|l\right) & =\int_{0}^{T_{obs}}\mbox{d}\chi^{2}\,\frac{1}{2^{l/2}\Gamma\left(l/2\right)}e^{-\chi^{2}/2}\left(\chi^{2}\right)^{-1+l/2}\nonumber \\
P\left(T<T_{obs}|l\right) & =\frac{\gamma(l/2,T_{obs}/2)}{\Gamma(l/2)}.\label{eq:chi2 CDF}\end{flalign}
In other words, it is the \emph{regularized incomplete} gamma function,
comprised of the \emph{lower incomplete }gamma function\[
\gamma(a,x)=\int_{0}^{x}\mbox{d}t\, t^{a-1}e^{-t}\]
and the \emph{complete }gamma function\[
\Gamma(a)=\int_{0}^{\infty}\mbox{d}t\, t^{a-1}e^{-t}.\]
This is true even though the individual observations in a run are
not normally distributed, but according to the \emph{half-normal }distribution,
since they are required to be successes. In fact, if $X_{i}$ is a
random variable distributed according to a
standard normal distribution limited
to the domain $\left[a_{i},\, b_{i}\right],\,\, a_{i},b_{i}\in\overline{\mathbb{R}}$,
the sampling distribution of\[
X_{1}^{2}+\dots+X_{l}^{2}\]
is given by the $\chi^{2}$- distribution with $l$ degrees of freedom
\eqref{eq:chi2 CDF}, regardless of the domains $\left[a_{i},\, b_{i}\right]$.
The proof follows the traditional lines by transforming to spherical
coordinates. It is then seen that the angular contributions (depending
on $a_{i},\, b_{i}$) are removed in the normalization, and the radial
behavior (independent of $a_{i},\, b_{i}$) is the $\chi^{2}$- distribution.
See, e.g., \cite[chap. 11]{stuart_kendalls_1994}  for details.

Now the derivation of the distribution of $T$ is completed, \eqref{eq:  intermed. result}
combined with \eqref{eq:chi2 CDF} give a complete specification that
can be implemented in just a few lines of code in \texttt{mathematica
\cite{wolfram_research_mathematica_2008}.} As an example, $P(T\ge T_{obs}|N=25)$
is plotted as a function of $T_{obs}$ in Fig. \ref{fig:exact solution}.
Since the number of partitions which contribute to $P\left(T<T_{obs}\left|N\right.\right)$
grows rapidly with $N$ (see appendix for details), we have to resort
to a Monte Carlo approximation of the $p-$value for $N\gtrsim80$.
Note that the Monte Carlo output also serves as a valuable cross check
with the exact solution for small $N$. We now briefly describe the
Monte Carlo algorithm:
\begin{enumerate}
\item Fix a number of experiments, $K$, and the number of observations,
$N$, in each experiment.
\item Generate $K\cdot N$ standard normal variates. \label{step:variates}
\item In each of the $K$ experiments, find the largest $\chi_{run}^{2}$
of any success run. This is $T_{obs,\, j}$ for the experiment $j,\, j=1\dots K$.
Filter out all experiments that contain no success.
\item \label{step:p-estimate}
Let $L$ denote the number of experiments in which $T_{obs,j} \ge T_{obs}$.
Then estimate the $p$-value, $P(T \ge T_{obs}|N)$ as $p\approx \frac{L}{K}$.
\end{enumerate}

We estimate the uncertainty on $p$ as obtained in step \ref{step:p-estimate}
from a Bayesian point of view. The sampling can be seen as a Bernoulli process,
with a constant chance of $p$ in each trial $j$ that $T_{obs,j} \ge T_{obs}$.
Assuming a uniform prior on $p$, the posterior then becomes
\begin{equation}
 P\left( p | L,K \right) = \frac{(K+1)!}{L!(K-L)!}p^L \left( 1-p \right)^{K-L}
\end{equation}
with the mode at $p=L/K$. Let $\langle \cdot \rangle$ denote the expectation value
under the posterior, then the variance of $p$ is
\begin{equation}
  \frac{\langle p \rangle \left( 1- \langle p \rangle \right)}{K+3}
\end{equation}
Thus for large $K$, the variance falls off as $1/K$.

As discussed in the introduction, we can define another statistic, call it $T^f$, analogous to $T$,
but now for the \emph{failure} runs instead of the success runs.
$T^f$ also tests the model's ability to reproduce the data.
In the algorithm indicated above, the same variates
obtained in step \ref{step:variates} can be used to calculate $p$-values for $T^f$.
One simply considers the largest $\chi^2_{run}$ of any failure run and filters
out all experiments with no failure.
Due to the symmetry of the Normal distribution, we have
\begin{equation}
\label{eq:succEQfail}
 P\left(T<T_{obs}|N\right)=P\left(T^f<T_{obs}|N\right)
\end{equation}

Given the set of samples $\{ T_{obs,j}\}$, we can construct the
\emph{empirical cumulative distribution function} (ECDF) \cite[chap. 25.3]{cramer_mathematical_1999}
for graphical display.
In Fig. \ref{fig:exact solution},
we show the Monte Carlo results ($1-\mbox{ECDF}(T_{obs}), K=10000, N=25$) for success
runs (green), and failure runs (red) and finally the exact results (blue) for $N=25$ data points.

For practical use, the critical values of $T$ for three often used
confidence levels $\alpha=5\%,\,1\%,\,0.1\%$ are presented in Table
\ref{tab:Critical-values-of}. Note that for fixed $\alpha$, the
critical values vary approximately linearly with $\log N$ \[
T_{crit}\left(N|\alpha\right)\sim c\cdot\log N+b\left(\alpha\right).\]
The slope $c$ appears to be nearly independent of $\alpha$. In Fig.
\ref{fig:critical values}, the following parameter values are chosen:\begin{flalign}
\alpha=0.05 & \Rightarrow c=2.8,\,\, b=2.5\nonumber \\
\alpha=0.01 & \Rightarrow c=2.9,\,\, b=6.1\label{eq:line parameters}\\
\alpha=0.001 & \Rightarrow c=3.0,\,\, b=11.6\nonumber \end{flalign}

\begin{figure}
\includegraphics{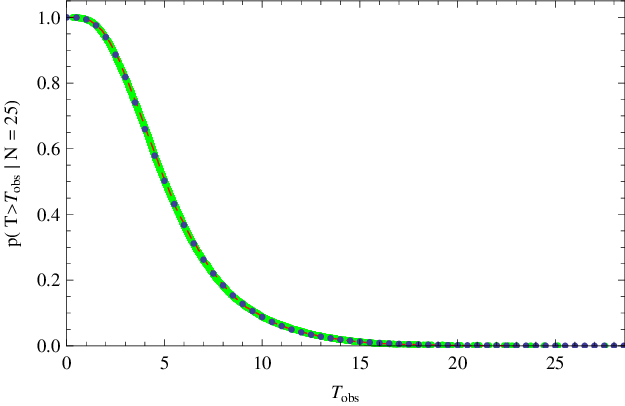}

\caption{$p-$value for the runs test statistic $T$ and $N=25$ observations.
The Monte Carlo results for successes (green) and
failures (red dashed) with $K=10000$ generated experiments
are in excellent agreement with the exact results (blue
dotted) using \eqref{eq:  intermed. result}, \eqref{eq:chi2 CDF}.
}
 \label{fig:exact solution}
\end{figure}

\begin{table}
\begin{center}

  \label{tab:Critical-values-of}

\begin{tabular}{c|ccccccc}
\hline
\textbf{$\mathbf{N}$} & 5 & 10 & 25 & 50 & 100 & 500 & 1000\tabularnewline
\hline
\textbf{$\boldsymbol{\alpha=0.05}$} & 6.8 & 8.8 & 11.5 & 13.4 & 15.3 & 19.8 & 21.6\tabularnewline
\textbf{$\boldsymbol{\alpha=0.01}$} & 10.4 & 12.8 & 15.7 & 17.7 & 19.7 & 24.4 & 25.9\tabularnewline
\textbf{$\boldsymbol{\alpha=0.001}$} & 15.5 & 18.3 & 21.6 & 23.8 & 25.6 & 29.9 & 32.0\tabularnewline
\hline
\end{tabular}

\smallskip{}

\caption{
Critical values of $T_{obs}$ at the $\alpha=5\%,\,1\%,\,0.1\%$ level
as a function of $N$.
Up to $N=50$ these are found from the exact solution. For larger
$N$, the critical values are estimated from the Monte Carlo approximation
using $K=10^{5}$ simulated experiments and linear interpolation of
$p\left(T_{obs}\right)$ based on the points $\left(T_{obs,\, j},\, p\left(T_{obs,\, j}\right)\right),\, j=1\dots K$.
}

\end{center}
\end{table}

\begin{figure}
\includegraphics{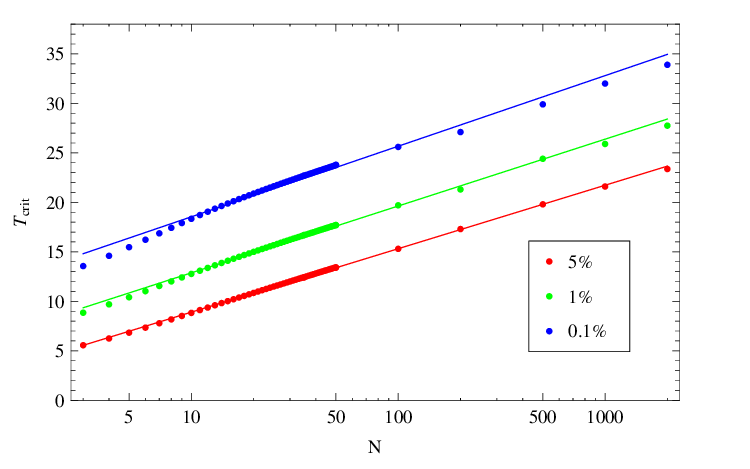}

\caption{Critical values of $T_{obs}$ at the $\alpha=5\%,\,1\%,\,0.1\%$ level.
$T_{crit}$ scales approximately linearly with $\log N$. The slope
is nearly independent of $\alpha$. \label{fig:critical values}}

\end{figure}

\section{Example\label{sec:Example}}

Let us discuss an example that frequently arises in high energy physics to
study the power of significance tests based on $T$. For comparison, we  use
the classic $\chi^2$ test statistic.

Assume an experiment is conducted to observe the quantity $y=y(x)$.
The uncertainties are modeled as arising from a
normal distribution with known variance, then for each of the $N$ independent
observations
\begin{equation}
 y_i \sim \mathcal{N}(\mu_i, \sigma_i^2).
\end{equation}

The purpose of the experiment is to decide whether the currently accepted
hypothesis $\mathcal{H}$ is sufficient to explain the data. The predictions
derived from $\mathcal{H}$ are given as
\begin{equation}
 \mu_i = f(x_i).
\end{equation}
In addition, assume there exists an extension to $\mathcal{H}$,
denoted by $\mathcal{H}_1$, whose
predictions are
\begin{equation}
 \mu_i = f(x_i)+g(x_i).
\end{equation}
Typically the extra contribution $g(x)$ is significant
in a narrow region only. For concreteness, we assume
it is a localized peak of the Cauchy-Lorentz form
with location parameter $\beta$ and scale
parameter $\gamma$
\begin{equation}
 g(x) = A \cdot \left( 1 + \frac{(x-\beta)^2}{\gamma^2} \right)^{-1}.
\end{equation}
The magnitude of the extra contribution is defined by $A$.
Three cases are to be distinguished.
For $A \to 0, \mathcal{H}_1 =\mathcal{H}$. For fixed
confidence level $\alpha$,
tests based on $T$ and $\chi^2$ reject $\mathcal{H}$ with
the nominal probability $\alpha$.

For $A \to \infty$, $\mathcal{H}$ is rejected with
probability 1 for either statistic.
In the most interesting region, $A$ not too small and
not too large, we study the rejection power of $T$ and $\chi^2$
by simulating experiments under $\mathcal{H}_1$. We then analyze the
data under $\mathcal{H}$ and estimate the power
as the fraction of times
the $p$-value is found in the rejection region defined by the
confidence level $\alpha= 0.05$.
We simulate an ensemble of 10000 experiments with $N=10$ draws from
$\mathcal{H}_1$
with $x_i=i, i=1 \dots 10$ and
parameters $\beta = 5.5, \gamma = 2$ fixed for different
values of $A$.
Without loss of generality, we choose
$f(x)=0, \sigma_i = 1$.
The numerical  results  are shown in Figure \ref{fig:Power}
as a function of $A$. The power of $T$ equals the power of
$\chi^2$ for $A=0$ and $ A \gg 1$ as expected. In the intermediate
region, the power of $T$ significantly exceeds that of $\chi^2$.
Similar results are obtained when keeping $A$ fixed and varying
$\gamma$ instead.

Moreover, if we choose a distribution with light tails (e.g.\ a normal
distribution) for $g(x)$ instead of the heavy-tailed Cauchy distribution,
the qualitative results are unaffected. The power of $T$ is  larger
than the power of $\chi^2$ for the alternative $\mathcal{H}_1$.
For  medium sized $g(x)$, the difference can reach up to 40\%.

\begin{figure}
\includegraphics{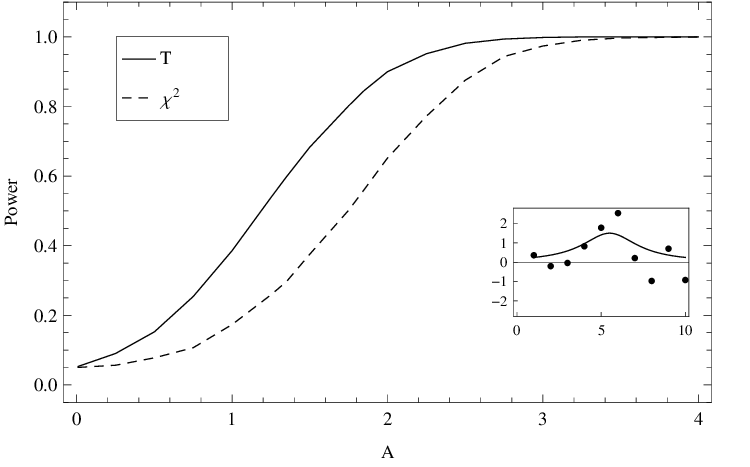}

\caption{
Power of statistics $T$ and $\chi^2$ in rejecting the null hypothesis of
normality around zero mean at the 5\% confidence level. For fixed $A$,
10000 experiments, each of sample
size 10, have been generated from a normal distribution with variance one. The mean
of sample $i, i=1 \dots 10$
is distributed according to a Cauchy distribution
$ A \cdot \left( 1 + \frac{(i-5.5)^2}{3^2} \right)^{-1}$. A sample data set ($A=1.5$)
is shown in the inset.
The curves show the
power as a function of the amplitude $A$.
}
\label{fig:Power}
\end{figure}

\section{Discussion\label{sec:Conclusion}}

We have introduced the test statistic $T$ and calculated its distribution
for the case of a sequence of independent observations, each following
a normal distribution with known mean and variance.   Implementing
the algorithm to calculate critical values of $T$ for the various
confidence levels is straightforward, but the execution time grows
rapidly with the number of observations $N$. Hence a Monte Carlo
scheme is recommended to calculate critical values for $N\gtrsim80$,
yielding results in reasonable time even for $N\gtrsim1000$, thus
covering virtually the whole range of interest relevant to everyday
problems. We have verified that the Monte Carlo results agree well
with exact results for small $N$.

We have demonstrated the usefulness of $T$ and recommend its usage for
hypothesis testing especially against alternatives with
additional local peaks.

The more common problem in data analysis is to consider a composite
hypothesis: in a first step free parameters of the model are estimated
from the data ({}``fit'') and in the second step predictions, based
on the fitted parameters, and observations are compared ({}``goodness
of fit''). With most test statistics the effect of fitted parameters
on the sampling distribution of the statistic is not analytically
known. The only notable exception to this rule is the $\chi^{2}$
statistic: for $k$ parameters extracted from maximizing the likelihood
of $N$ normal observations, the number of degrees of freedom is $N-k$,
instead of $N$ in case all parameters are known a priori. Unfortunately,
this cannot be extended to the runs statistic $T$ considered here.
However what we can do is to simulate data sets using a Monte Carlo
approach, and study the approximate numerical distribution of $T$.
For the simplest case of a straight line and a maximum likelihood
fit to 10 data points, the results are shown in Fig. \ref{fig:Distribution-of-fit}.
It is evident that $p\left(T\right)$ drops to zero much more sharply
for the fitted data (green=successes, red=failures) than for the exact
results with no parameters fitted (blue). Accordingly, the critical
values for fitted $T$ at level $\alpha=5\%,\,1\%,\,0.1\%$ are $T_{crit}=6.0,8.5,12.4$.
In general, the qualitative effect of fitting parameters but pretending
that they were known before the data was taken is that the $p$-value
is not distributed uniformly. Instead, its distribution is biased
towards $p=1$, leading to conservative decisions. The quantitative
effect depends on  the number of observations and parameters, the
maximization condition determining the best fit parameters (likelihood,
posterior ...) and possibly other effects.

Through Monte Carlo approximations the use of the runs statistic $T$
can be further generalized to the important class of problems involving
asymmetric uncertainties like Binomial or Poisson distributions. All
that needs to be changed is the weight of individual runs. As a starting
point one could define $T$ as the smallest  probability (density)
of any run, $T=\min_{j}P\left(A_{j}\left|\mathcal{H}\right.\right)$.
Numerically the distribution of $T$ is then found in analogous fashion
to the algorithm described in the caption of Fig. \ref{fig:Distribution-of-fit}.
An implementation of this algorithm is scheduled to be included in
a future release of \emph{BAT}, the \emph{Bayesian Analysis Toolkit}
\cite{Caldwell20092197}. BAT is a C++ library based on the \emph{Markov
Chain Monte Carlo} approach which offers routines for fitting, limit
setting, goodness of fit and more. Using the Metropolis algorithm
\cite{metropolis__1953} it is possible to simulate the data sets
needed for approximate $p$-value calculations.

\begin{figure}
\includegraphics{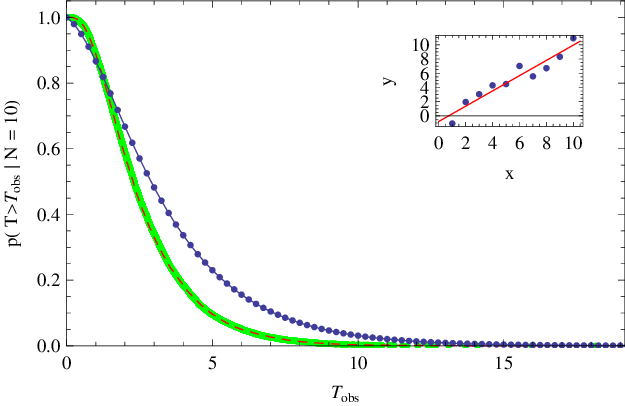}

\caption{
Distribution of runs test statistic $T$ with and without fitted parameters.
The Monte Carlo results for successes (green) and
failures (red dashed) are obtained from $K=10000$
generated experiments. Each data set consists of $N=10$ data points
$\left(x_{i},y_{i}\right)$, where the $y_{i}$ are normally distributed
around a straight line of unit slope and zero intercept, $y_{i}\sim\mathcal{N}\left(\mu=1\cdot x_{i}+0,\sigma^{2}=1\right)$.
Then a maximum likelihood fit is performed to extract the two parameters
of a straight line model $y=m\cdot x+b$ (see inset). Finally the
predictions are calculated from the fitted model, and $T_{obs}$ is
determined for each experiment. With the set of 10000 values of $T_{obs}$
the empirical CDF (ECDF) is computed, and $1-\mbox{ECDF}\left(T_{obs}\right)$
is plotted. For comparison the exact results (blue dotted) for $N=10$ using Eq. \eqref{eq:  intermed. result},
\eqref{eq:chi2 CDF} are shown. The effect of fitting is that $p\left(T\right)$
drops more sharply, hence the critical values are pushed towards smaller
$T$; e.g. at the 5\% level $T_{crit}=6.0$ (fit) vs $T_{crit}=8.8$
(no fit).
}
\label{fig:Distribution-of-fit}
\end{figure}

 \appendix
\renewcommand{\theequation}{A.\arabic{equation}}

\section*{Appendix}

\subsection*{Integer Partitions and Computational Complexity\label{sub:Integer-Partitions-and}}

We are now interested in the number of sequences, $\nu\left(N\right)$,
which need to be taken into account to calculate a $p$-value for
$T$, $P(T\ge T_{obs}|N)$, using \eqref{eq:  intermed. result}.
Put differently, $\nu\left(N\right)$ is the number of terms in
the multiple sum \begin{equation}
P(T\ge T_{obs}|N)=1-\sum_{r=1}^{N}\sum_{M=1}^{\min(r,N-r+1)}\sum_{\pi}\dots,\label{eq:schematic p}\end{equation}
where $\sum_{\pi}$ extends over all inequivalent sequences with $r$
successes distributed in $M$ success runs, see \eqref{def:equival relation}
and \eqref{eq: sum over partitions}. Since $\nu\left(N\right)$
determines  the number of steps needed to calculate the $p$-value
on a computer, knowing the form of the $N$-dependence aids in ascertaining
whether the computer can be expected to finish the calculation in
reasonable time. In the main result of this section, Proposition \ref{proposition:Let--denote},
$\nu\left(N\right)$ is essentially given by the number of integer
partitions. To begin with, we introduce the integer partitions and
illustrate with an example. The book \cite{andrews_theory_1998} by Andrews
is a good reference devoted entirely to partitions.
\begin{defn}
Let $\part N$ denote the number of partitions of the integer $N$
into a sum of one or more positive integers. For consistency it is
useful to define $\part{0}\equiv1$. Let $\part{N,\, k}$ denote the
number of partitions of $N$ into exactly $k$ addends and finally
let $\partLE{N,\, i}$ denote the number of partitions of $N$ into integers
of at most size $i$, with $\partLE{0,\, i}\equiv1$.
\label{def:part N}
\end{defn}
\begin{example}
The integer 5 can be written in $\part5=7$ different ways:

\begin{eqnarray}
5 & = & 5\\
 & = & 4+1\\
 & = & 3+2\\
 & = & 3+1+1\label{eq:3a}\\
 & = & 2+2+1\label{eq:3b}\\
 & = & 2+1+1+1\\
 & = & 1+1+1+1+1
\label{exa:The-integer-5}
\end{eqnarray}
One can see that 5 can be decomposed as a sum of exactly three non-zero
integers in two ways (Eq. \eqref{eq:3a} and \eqref{eq:3b}), thus $\part{5,3}=2.$ Furthermore,
the number of ways to partition 5 into addends less than 3 is $\partLE{5,\,2}=3$
(Eq. \eqref{eq:3b}-\eqref{exa:The-integer-5}).
\end{example}
The three partition numbers just defined are obviously closely connected,
we shall need the following relations; elementary proofs based on
Ferrer's diagrams can be found in the books by Andrews \cite[chap. 1]{andrews_theory_1998}
and Knuth \cite[chap. 7.2.1.4]{knuth_art_2005}.
\begin{fact}
Assuming $N\ge1$, Def. \ref{def:part N} yields:\[
\part N=\sum_{r=1}^{N}\part{N,r}\]
\begin{equation}
\partLE{N,r}=\sum_{M=1}^{r}\part{N,M}\label{fac:Using-definition-}\end{equation}
\[
\partLE{M,r-M}=\part{r,r-M}\]

\begin{equation}
\part{N}=\sum_{r=0}^{N-1}\partLE{r,N-r}\label{lem:Furthermore--and}\end{equation}
\end{fact}
\begin{prop}
\label{proposition:Let--denote}Let $\nu\left(N\right)$ denote
the number of inequivalent Bernoulli sequences of length $N$, where
the probability of a success is $\frac{1}{2}$ in each trial and the
equivalence relation is defined in \eqref{def:equival relation}.
Then \begin{eqnarray}
\nu\left(N\right) & \equiv & \sum_{r=1}^{N}\sum_{M=1}^{\min(r,\, N-r+1)}\part{r,M}\\
 & = & \part{N+1}-1\end{eqnarray}
\end{prop}
\begin{proof}
We start from the right hand side of the proposition using \eqref{lem:Furthermore--and}:
\begin{flalign}
\part{N+1}-1 & =-1+\sum_{r=0}^{N}\partLE{r,\, N+1-r}\\
 & =\sum_{r=1}^{N}\partLE{r,\, N+1-r}.\end{flalign}
Now using \eqref{fac:Using-definition-}:\[
\part{N+1}-1=\sum_{r=1}^{N}\sum_{M=1}^{N-r+1}\part{r,\, M}.\]
But we know that we cannot partition $r$ successes into more than
$r$ success runs, so $\part{r,\, M>r}=0$, hence

\begin{eqnarray}
\part{N+1}-1 & = & \sum_{r=1}^{N}\sum_{M=1}^{\min\left(r,\, N-r+1\right)}\part{r,\, M}\\
 & = & \nu\left(N\right).\end{eqnarray}

\end{proof}
Since $\part{r,M}$ represents the number of elements in $\sum_{\pi}$
of \eqref{eq:schematic p}, $\nu\left(N\right)$ is the exact number
of sequences that contribute to $P(T\ge T_{obs}|N)$. We can \emph{approximate}
$\nu\left(N\right)$ by employing the asymptotic expression of
$\part{N}$ for  large  $N$ first derived by
Hardy and Ramanujan \cite{hardy_asymptotic_1918}:

\[
\part{N}\sim\frac{\exp\left(\pi\sqrt{2/3\cdot N}\right)}{4\sqrt{3}N}.\]
Hence, for large $N$, $\nu\left(N\right)$ grows nearly exponentially.
\begin{cor}
For large $N$, $\nu\left(N\right)$ is approximately given by
\begin{equation}
\nu\left(N\right)\sim\frac{\exp\left(\pi\sqrt{2/3\cdot\left(N+1\right)}\right)}{4\sqrt{3}\left(N+1\right)}\label{eq:asym partition number}\end{equation}

\end{cor}
This implies that for large $N$ (say $N=1000$), in equations \eqref{eq: sum over r,M},
\eqref{eq: sum over partitions} the sum is over more partitions ($\nu\left(N\right)=2.5\times10^{31}\approx2^{104}$)
than a current 64-bit desktop computer could even address in memory.
 In practice the exact evaluation of $P(T\ge T_{obs}|N)$ becomes
too slow already for $N\gtrsim80$ where $\nu\left(80\right)=1.8\times10^{7}$.
In contrast a Monte Carlo solution based on sampling a large number
of batches, $K$, each with $N$ pseudo random numbers is much faster:
its computational complexity is $\mathcal{O}\left(K\cdot N\right)$.





  \bibliographystyle{elsarticle-num}
\bibliography{runs}







\end{document}